\documentclass[11pt,letterpaper]{article}

\title{Enumerating Restricted Dyck Paths with Context-Free Grammars}
\author{AJ Bu \and Robert Dougherty-Bliss}
\date{\today}

\usepackage[T1]{fontenc}
\usepackage{lmodern}
\usepackage[pdftex,margin=1in]{geometry}
\usepackage{enumerate}
\usepackage{amsmath}
\usepackage{amssymb}
\usepackage{fancyhdr}
\usepackage{lastpage}
\usepackage[alwaysadjust]{paralist}
\usepackage{amsthm}
\usepackage{graphicx}

\usepackage{wasysym}
\usepackage{xcolor}
\usepackage{listings}
\usepackage{hyperref}

\newcommand{\Z}{\mathbb{Z}}

\newtheorem{Thm}{Theorem}
\newtheorem{Prop}[Thm]{Proposition}

\theoremstyle{definition}
\newtheorem{Def}[Thm]{Definition}

\theoremstyle{remark}

\theoremstyle{definition}

\theoremstyle{definition}

.

\newenvironment{Proof}{\noindent\textbf{Proof.}}{\qed}

\lstdefinelanguage{Maple}%
{morekeywords={and,assuming,break,by,catch,description,do,done,%
		elif,else,end,error,export,fi,finally,for,from,global,if,%
		implies,in,intersect,local,minus,mod,module,next,not,od,%
		option,options,or,proc,quit,read,return,save,stop,subset,then,%
		to,try,union,use,uses,while,xor},%
	sensitive=true,%
	morecomment=[l]\#,%
	morestring=[b]",%
	morestring=[d]"%
}[keywords,comments,strings]%


\fancypagestyle{fancyfirst}{%
  \pagestyle{fancy}%
  \chead{\small Dyck Paths}%
}

\begin{document}

\maketitle


\begin{abstract}
    The number of Dyck paths of semilength $n$ is famously $C_n$, the $n$th Catalan number. This fact follows after noticing that every Dyck path can be uniquely \emph{parsed} according to a context-free grammar. In a recent paper, Zeilberger showed that many \emph{restricted} sets of Dyck paths satisfy different, more complicated grammars, and from this derived various generating function identities. We take this further, highlighting some combinatorial results about Dyck paths obtained via grammatical proof and generalizing some of Zeilberger's grammars to infinite families.
\end{abstract}

\section{Introduction}%
\label{sec:introduction}

\noindent As Flajolet and Sedgewick masterfully demonstrate in their seminal
text, \emph{Analytic Combinatorics} \cite{flajolet}, mathematicians have
occasionally borrowed the study of formal languages from computer science and
linguistics for combinatorial reasons. Many combinatorial classes can be
reinterpreted as languages generated by certain grammars, and these grammars
often make writing down generating functions, another favorite combinatorial
tool, routine.

For example, consider the well-known \emph{Dyck paths}. A Dyck path is a finite
list of $+1$'s and $-1$'s whose partial sums are nonnegative, and whose sum is
$0$. We will write $U$ (up) for $+1$ and $D$ (down) for $-1$. Thus, the
following are all Dyck paths:
\begin{align*}
    &UUDD \\
    &UDUD \\
    &UUUDUDDD
\end{align*}
A Dyck path must have even length, since ``number of $U$'s'' equals ``number of
$D$'s.'' For this reason, we often refer to Dyck paths of \emph{semilength $n$}
(length $2n$).

It is a famous result that the number of Dyck paths of semilength $n$ equals
the $n$th \emph{Catalan number}:
\begin{equation*}
    C_n = \frac{1}{n + 1} {2n \choose n}.
\end{equation*}
There are many proofs of this fact, but here is a \emph{grammatical} proof.

Let $\mathcal{P}$ denote the set of all Dyck paths. Then, $\mathcal{P}$ is
generated by the unambiguous, context-free grammar
\begin{equation*}
    \mathcal{P} = \epsilon \ \cup \ U\mathcal{P}D\mathcal{P},
\end{equation*}
where $\epsilon$ denotes the empty string. In words, a path is either empty or
begins with a $U$, is followed by a Dyck path (shifted to height $1$), a $D$,
then another Dyck path. This is a unique parsing of all Dyck paths.

Given a set of objects $E$ each with a nonnegative integer size,
let $GF(E) = \sum_{k \geq 0} |E(k)| z^k$ be a formal generating function,
where $|E(k)|$ is the number of objects of size $k$ in $E$. The main result
about formal grammars is that, in an unambiguous context free grammar,
\begin{equation*}
    GF(A \cup B) = GF(A) + GF(B)
\end{equation*}
and
\begin{equation*}
    GF(AB) = GF(A) GF(B),
\end{equation*}
where the ``sizes'' of the grammar are the lengths of the words it generates.

In our case, if $P(z)$ is the generating function for the number of Dyck paths
of semilength $n$, then this grammar implies
\begin{align*}
    P(z) &= GF(\epsilon) + GF(U\mathcal{P}D\mathcal{P}) \\
         &= 1 + z P(z)^2.
\end{align*}
(There is exactly one empty Dyck path [which has semilength 0], and the
presence of $U$ and $D$ increases the semilength by $1$.) The generating
function $C(z)$ for the Catalan numbers \emph{also} satisfies
\begin{equation*}
    C(z) = 1 + z C^2(z),
\end{equation*}
and since there are only two possible solutions, it is not hard to see that
$P(z) = C(z)$.

The grammatical technique offers a unifying framework: Devise a grammar and you
get an equation. Sometimes the equations turn out to be well-known. Other times
they are complicated messes. The enumeration of all Dyck paths is one
application of this framework, and here we want to demonstrate others. In
particular, we will give grammatical proofs of several combinatorial facts
about \emph{restricted} Dyck paths, and also establish several infinite
families of grammars in closed form.

First, let us define the restrictions we shall consider.

\begin{Def}
    Given a Dyck path, the \emph{height} of the path at position $k$ is the
    partial sum of the path after its $k$th term. A \emph{peak} of a Dyck path
    at height $h$ (or simply ``at $h$'') is the bigram $UD$ where the height of
    the path after the $U$ is $h$. Similarly, a \emph{valley} occurs at the
    bigram $DU$, and its height is analogously defined. The empty path has,
    by convention, a peak at $0$ but no valley.

    Given a sequence of steps $L$, define $L^n$ to be the repetition of $L$ $n$ times. (For example, $U^2 = UU$ and $(UD)^3 = UDUDUD$.)

    A Dyck path has an \emph{up-run of length $n$} provided that it contains at least one $U^n$ that is not preceded nor followed by $U$. Similarly, it contains a \emph{down-run of
    length $n$} provided that it contains at least one $D^n$ that is neither preceded nor followed by $D$.
\end{Def}

We are generally considered with Dyck paths whose peaks and valley heights
avoid certain sets, and whose up-run and down-run lengths avoid certain sets,
and combinations of the four conditions. We will, for example, discuss the set
of all Dyck paths whose peak heights avoid $\{2, 4, 6, \dots\}$ and have no
up-run of length greater than $2$.

When a set $\mathcal{P}$ of Dyck paths has been specified and used in an
expression, such as $\mathcal{P} = U \mathcal{P} D \mathcal{P}$, it is
shorthand for ``any (possibly vertically shifted) Dyck path from
$\mathcal{P}$.''

\begin{Def}
    For arbitrary sets of positive integers $A$, $B$, $C$, and $D$, let $P(A,
    B, C, D)$ be the set of Dyck Dyck paths whose peaks heights avoid $A$,
    whose valleys avoid $B$, whose up-run lengths avoid $C$, and whose down-run
    lengths avoid $D$. Let $P_{A, B, C, D}(z)$ be be the generating function
    for the number of Dyck paths of semilength $n$ in $P(A, B, C, D)$.
\end{Def}

Some of these sets have been studied. In \cite{peart}, Peart and Woan provide a
continued-fraction recurrence for the generating functions $P_{\{k\},
\emptyset, \emptyset, \emptyset}(z)$. In \cite{eu}, where Eu, Liu, and Yeh take
this idea further and express $P_{A, \emptyset, \emptyset, \emptyset}(z)$ as a
finite continued fraction whenever $A$ is finite or an arithmetic progression.
In \cite{z}, Zeilberger presents a rigorous experimental method to derive
equations for $P_{A, B, C, D}(z)$ when the sets involved are finite or
arithmetic progressions. Proving ``by hand'' some of Zeilberger's interesting
discoveries \emph{ex post facto} was a motivation for the present work.
We generalize some of Zeilberger's results to infinite families which are
likely out of reach for symbolic methods.

Our results include several explicit grammars (and therefore generating
function equations) for infinite families of the sets $A$ and $B$, and also
grammatical proofs of several interesting special cases suggested in \cite{z}.
Many of these---any grammars referencing restrictions on up- or down-runs---are
not in \cite{eu}. Some of our results are suggested in the OEIS \cite{oeis};
see, for example, A1006 (Motzkin numbers) and A004148 (generalized Catalan
numbers).

The remainder of the paper is organized as follows. Section~\ref{sec:results}
presents some results discovered by experimentation with software from \cite{z}
and proven with grammatical methods. Section~\ref{sec:grammar} presents some
infinite families of explicit grammars. Section~\ref{sec:conclusion} offers
some concluding remarks about the limitations of grammars.

\section{Combinatorial results}%
\label{sec:results}

In this section we will present a number of results with grammatical proofs. We
will often abuse notation and use one symbol---$P$, for example---to
simultaneously denote a set of Dyck paths, a generating function, and a
non-terminal symbol in a formal grammar.

\begin{Prop}
    The number of Dyck paths of semilength $n$ whose peak heights avoid $\{2r +
    3 \mid r \geq 0\}$ and whose up-runs are no longer than $2$ is $1$ when $n
    = 0$, and $2^{n - 1}$ when $n \geq 1$.
\end{Prop}

\begin{proof}
    Let $P$ be the set of all such Dyck paths, and $Q$ the set of all Dyck
    paths which avoid peaks in $\{2r + 2\}$ and up-runs longer than $2$. Note
    that $P$ and $Q$ satisfy the following grammar:
    \begin{align*}
        P &= \epsilon \ \cup \ UDP\ \cup\ UUD Q D P \\
        Q &= \epsilon \ \cup \ UD Q.
    \end{align*}
    This implies the following system of equations:
    \begin{align*}
        P &= 1 + z P + z^2 Q P \\
        Q &= 1 + z Q.
    \end{align*}
    Thus $Q(z) = (1 - z)^{-1}$ (the only path in $Q$ of semilength $n$ is
    $(UD)^n$) and
    \begin{equation*}
        P(z) = \frac{1 - z}{1 - 2z}.
    \end{equation*}
    Therefore $[z^0] P(z) = 1$ and $[z^n] P(z) = 2^{n - 1}$.
\end{proof}

\begin{Prop}
    The number of Dyck paths of semilength $n$ whose peak heights avoid $\{2r +
    3 \mid r \geq 0\}$ and whose up-runs are no longer than $3$ equals the $(n
    + 1)$th generalized Catalan number $G_{n + 1}$, defined by
    \begin{align*}
        G_0 &= 1 \\
        G_1 &= 1 \\
        G_{n + 2} &= G_{n + 1} + \sum_{1 \leq k < n + 1} G_k G_{n - k}.
    \end{align*}
\end{Prop}

\begin{proof}
    Let $P$, $O$, and $E$ be the set of all Dyck paths with up-runs no longer
    than $3$, and whose peak heights avoid $\{2r + 3 \mid r \geq 0\}$, $\{2r +
    2 \mid r \geq 0\}$, and $\{2r + 1 \mid r \geq 0\}$, respectively. Observe
    that $P$, $O$, and $E$ satisfy the following grammar:
    \begin{align*}
        P &= \epsilon \ \cup \ UD P\ \cup \ UUDODP \\
        O &= \epsilon \ \cup \ UD O \ \cup \ UUU D O D E D O \\
        E &= \epsilon \ \cup \ UUD O D E
    \end{align*}
    This grammar implies the following equations:
    \begin{align*}
        P &= 1 + zP + z^2 OP \\
        O &= 1 + z O + z^3 EO^2 \\
        E &= 1 + z^2 OE.
    \end{align*}
    This system has two possible solutions for $P$, but only one is holomorphic
    near the origin, namely
    \begin{equation*}
        P(z) = \frac{2}{1 - z - z^2 + (z^4 - 2z^3 - z^2 - 2z + 1)^{1/2}}.
    \end{equation*}
    The generating function $G(z)$ for the generalized Catalan numbers is
    well-known to be
    \begin{equation*}
        G(z) = \frac{1 - z + z^2 - \sqrt{1 - 2z - z^2 - 2z^3 + z^4}}{2z^2},
    \end{equation*}
    and it is routine to verify that $G(z) = zP(z) + 1$. Therefore $G_{n + 1} =
    [z^n] P(z)$ for $n \geq 0$.
\end{proof}

The following proposition is concerned with \emph{Motzkin numbers} (see A1006
in the OEIS and \cite{motzkin}). A Motzkin \emph{path} is like a Dyck path, but
includes a ``sideways'' step $S$ which does not change the height. The $n$th
Motzkin number $M_n$ is the number of Motzkin paths of length $n$. The
generating function $M = M(z)$ for $M_n$ satisfies the quadratic equation
\begin{equation*}
    M = 1 + z M + z^2 M^2.
\end{equation*}

There are numerous bijections between Motzkin paths and various restricted
classes of Dyck paths. Such bijections are often variations of the ``folding''
map
\begin{align*}
    UD &\mapsto S \\
    DU &\mapsto S \\
    UU &\mapsto U \\
    DD &\mapsto D,
\end{align*}
which in general is not injective, but many restrictions on Dyck paths
\emph{make} it injective. For example, this idea shows that the Dyck paths of
semilength $n$ with no up-runs longer than $2$ are in bijection with the
Motzkin paths of length $n$. We offer a grammatical proof of this fact.

\begin{Prop}
    The number of Dyck paths of semilength $n$ which avoid up-runs of length
    $3$ or more equals the $n$th Motzkin number $M_n$.
\end{Prop}

\begin{proof}
    Let $P$ be the set of such paths. A grammar for $P$ is
    \begin{equation*}
        P = \epsilon\ \cup\ UUD P D P\ \cup\ U D P.
    \end{equation*}
    Our grammar implies that
    \begin{equation*}
        P = 1 + zP + z^2 P^2.
    \end{equation*}
    This is the same equation satisfied by the Motzkin generating function, and
    it is easy to check that $P(z) = M(z)$.
\end{proof}

\begin{Prop}
    Consider the set of Dyck paths such that no peak or valley has positive,
    even height. The numbers of such paths of semilength $2n$ and $2n+1$ are
    ${2n - 1 \choose n}$ and ${2n \choose n}$, respectively.
\end{Prop}

\begin{enumerate}[\hspace{1 cm}] 
	\item[1.] For any Dyck path, the first step must be up and the last step must be down to avoid negative height.

	\item[2.] The parity of the height at the $k$th step clearly equals the parity of $k$.
	\begin{enumerate}[\hspace{1 cm}] 
		\item[a.]  if the $2k^{th}$ step is up, the $(2k+1)^{th}$ step must also be up to avoid a peak with even height.  

		\item[b.] If the $2k^{th}$ step is down and the height is not zero, then the next step must also be down to avoid a valley with positive even height. 

		\item[c.] if the $2k^{th}$ step is down and the height is $0$, then the next step must be up to avoid a negative height.
	\end{enumerate}
\end{enumerate}

\begin{Proof}

    Let $D$ be the set of Dyck Paths with semi-length $2n+1$ such that no peak-height and no valley-height is a positive even number. By our note above, it is clear that each $d\in D$ is defined by the
    direction of its even numbered steps and its height at each of these steps,
    excluding its final step which must be down.  Since $d$ has length $4n+2$,
    there are $2n$ of these steps.

    Let $W$ be the set of all walks of length $2n$ with steps up and down such
    that the starting and ending height are both $0$. Note that $W$ consists of all permutations of $n$ up steps and $n$
    down-steps, and therefore $|W|=\binom{2n}{n}$. We define a bijection between $D$ and $W$ as follows.

    Working our way through $k=1\dots2n$: Given a walk $w \in W$, the height going from $-1$ to $0$ or from $0$ to
    $-1$ at the $k^{th}$ step corresponds to a Dyck path whose height is zero
    at the $2k^{th}$ step. (So the $2k^{th}$ step is down and the $(2k+1)^{th}$
    step is up). Otherwise, an increase in the absolute value of the height at the $k^{th}$
    step corresponds to a Dyck path whose $2k^{th}$  and $(2k+1)^{th}$ steps
    are both up, and a decrease corresponds to a Dyck path whose $2k^{th}$  and
    $(2k+1)^{th}$ steps are both down.

    Note that  the height can only go from $-1$ to $0$ at even steps and from
    $0$ to $-1$ at odd steps, so this is clearly injective. Its inverse is
    also injective since, when determining the $k^{th}$ step of the walk, given
    a Dyck path, we already know all the preceding steps in that walk.

Now, let $D$ be the set of Dyck Paths with semi-length $2n$ such that no peak-height and no valley-height is a positive even number. As in the previous proof, $d\in D$ is defined by the direction of its even numbered steps and its height at each of these steps, excluding its final step which must be down.  Since $d$ has length $4n$, there are $2n-1$ of these steps. 

Let $W$ be the set of all walks of length $2n-1$ with steps up and down such that the starting  height is $0$ and ending height is $-1$. Note that $W$ consists of all permutations of $n-1$ up steps and $n$ down-steps, and therefore $|W|=\binom{2n-1}{n}$. 

Here, we can define a bijection between $D$ and $W$ the same way that we defined it in the previous proof. Note that for $w\in W$, we start at $0$ and end at $-1$, so the number of steps away from the line between heights $0$ and $-1$ is still the same as the number of steps towards it, so its image ends at height $0$.  Moreover, there can never be more steps towards zero than away from zero (i.e. absolute value can never decrease more than it increases), so the image will never have negative height and thus is a Dyck path. As before, the image will never have a peak or valley height that is a positive even number.

In the other direction, $d\in D$ obviously maps to a walk of length $2n-1$ that starts at height $0$. Since $d$ starts and ends at height $0$, if we remove the first and last step of $d$ and split the remaining path into sub-paths of length 2, then $[1,1]$ appears the same number of times as $[-1,-1]$.  Thus, the ending height if the image of $d$ will either be $0$ or $-1$ (since the number of steps away from the line between $0$ and $-1$ equals the number of steps toward that line). Since $d$ has semi-length $2n-1$, $[-1,1]$ will occur an odd number of times in $d$. Therefore the image of $d$ will cross the line between $0$ and $-1$ an odd number of times, starting at $0$, and thus will end at $-1$.

Both maps are injective for the same reasons as before.
 
\end{Proof}

\section{Grammatical families}
\label{sec:grammar}

In this section we provide some explicit grammars for infinite families of
restricted Dyck paths. In many cases, such grammars
are guaranteed to exist. The reasoning in \cite{z} shows that, for
every set of Dyck paths whose peaks, valleys, and up- and down-runs avoid
specific arithmetic progressions, we may construct a finite, context-free
grammar which generates them. The method implied in \cite{z} to compute these
grammars gives no hint as to their \emph{form}, and this is what we try to
provide here.

Our first two results are about Dyck paths whose up-run lengths avoid a
fixed arithmetic progression $\{Ar + B \mid r \geq 0\}$. It turns out
that when $B < A$, there is a simple context-free grammar for such paths.
When $B \geq A$ the situation is more complicated, but we can derive a
``grammatical equation'' which again leads to a generating function.

\begin{Prop}
	Let $B < A$ be non-negative integers. The set $\mathcal{P}$ of Dyck
	paths whose up-run lengths avoid $\{Ar + B \mid r \geq 0\}$ has the
	unambiguous grammar
    \begin{equation*}
        \mathcal{P}= \bigcup_{\substack{0 \leq k < A \\ k \neq B}} U^k (D\mathcal{P})^k \cup U^A (\mathcal{P}D)^A \mathcal{P},
    \end{equation*}
    and therefore
    \begin{equation*}
        P(z) = \sum_{\substack{0 \leq k < A \\ k \neq B}} z^k P^k(z) + z^A P^{A + 1}(z),
    \end{equation*}
    where $P(z)$ is the weight-enumerator of $\mathcal{P}$.
\end{Prop}

\begin{Proof}
    The grammar clearly uniquely parses the empty path, so suppose that
    $P \in \mathcal{P}$ has length $n > 0$. Then $P$ starts with a
    up-run of length $k > 0$ for some $k \not\equiv B \mod A$. If
    $k < A$, then write $P = U^k D W$,
    where $W$ is a walk from height $k-1$ to height $0$ with the same
    restrictions on up-runs as $P$. For $0 \leq i < k - 1$, let
    $D_i$ indicate the down-step in $W$ which hits the height $i$ for
    the first time. Then
    \begin{equation*}
    	W = P_{k-1} D_{k-2} P_{k-2} D_{k-3}... P_1 D_0 P_0,
    \end{equation*}
    where $P_i$ is a Dyck path shifted to height $i$ with the same
    restrictions on up-runs as $P$. This uniquely parses $P$ into the case
    $U^k (D \mathcal{P})^k$ in the grammar.

    If the initial up-run has length $k \geq A$, then write
    $P = U^A W$,
    where $W$ is a walk from height $A$ to height $0$ whose up-run lengths avoid
    $\{Ar + B \mid r \geq 0\}$.
    By argument analogous to the previous paragraph, we can decompose
    $W$ as
    \begin{equation*}
        W=P_A D_{A-1} P_{A-1} D_{A-2}... P_1 D_0 P_0,
    \end{equation*}
    where $P_i \in \mathcal{P}$. Thus $W$ is of the form $(\mathcal{P} D)^A \mathcal{P}$, and this uniquely parses $P$ into the final case of the grammar.
    
    We have shown that $\mathcal{P}$ is contained in the language generated by this grammar,
    and it is easy to see that the first $k$ cases of the grammar are contained
    in $\mathcal{P}$. The final case, $U^A (\mathcal{P} D)^A \mathcal{P}$ is
    also contained in the grammar, because concatenating $U^A$ to the beginning of a path does not
    change the length any of the up-runs modulo $A$. The different cases are clearly disjoint, so the grammar is also unambiguous.
\end{Proof}

\begin{Prop}
	Let $A \leq B$ be nonnegative integers. The set $\mathcal{P}$ of
	Dyck paths avoiding up-run lengths in $\{Ar + B \mid r \geq 0\}$
    satisfies the ``grammatical equation''
	\begin{equation*}
    	\mathcal{P} \cup U^B (D\mathcal{P})^B =
    	    \bigcup_{0 \leq k < A }
    	        U^k (D \mathcal{P})^k\ \cup\ U^A (\mathcal{P}D)^A \mathcal{P},
	\end{equation*}
    and therefore
    \begin{equation*}
        P(z) + z^B P(z)^B = \sum_{0 \leq k < A} z^k P^k(z) + z^A P^{A + 1}(z),
    \end{equation*}
    where $P(z)$ is the weight-enumerator of $\mathcal{P}$.
\end{Prop}

Note that the right-hand side is nearly identical to the previous claim;
the difference being that we can get paths in $U^B (D\mathcal{P})^B$, which we will show below.

\begin{Proof}
    If $P$ is a path in $\mathcal{P}$, then we can uniquely parse $P$
    into a case of the right-hand side by the same argument given in the
    previous proposition.
    \begin{align*}
        U^B (D \mathcal{P})^B &= U^A U^{B - A} (D \mathcal{P})^B \\
          &= U^A \{U^{B - A} (D \mathcal{P})^{B - A}\} (D \mathcal{P})^A \\
          &= U^A [\{U^{B - A} (D \mathcal{P})^{B - A}\} D (\mathcal{P} D)^{A - 1}] \mathcal{P}.
    \end{align*}
    The expression in brackets, $U^{B - A} (D \mathcal{P})^{B - A}$, is in
    $\mathcal{P}$, which shows that
    $ U^B (D \mathcal{P})^B$ is contained in $U^A (\mathcal{P}D)^A \mathcal{P}$.

    Conversely, it remains to show that the left-hand side is \emph{all}
    that the right-hand side can generate.  $\bigcup_{0 \leq k < A }
    	        U^k (D \mathcal{P})^k$  is contained in $\mathcal{P}$ as in the previous proposition. For $W \in U^A(\mathcal{P} D)^A \mathcal{P}$, write
    \begin{equation*}
        W = U^A P_1 D \dots  P_A D P_{A+1}.
    \end{equation*}
    Let $\ell$ be the length of the initial up-run in
    $P_1$. If $\ell \not\equiv B \pmod A$, then $W$ contains no up-runs of lengths
    in $\{Ar + B \mid r \geq 0\}$ and is a path in $\mathcal{P}$. If $\ell \equiv B \pmod A$, then $\ell \leq B-A$. If $\ell<B-A$ then the initial run of $W$ has length less than $B$. Thus, $W$ contains no up-runs of lengths in $\{Ar + B \mid r \geq 0\}$. For $\ell=B-A$, let $D_i$ denote the first time $W$ steps down to height $i$ for $A<i< B$ and write
    \begin{align*}
         W &=U^A P_1 D \dots  P_A D P_{A+1}\\
            &=U^A (U^{B-A} D_{B-1} W_{B-1} \dots D_A W_A) D P_2 D \dots  P_A D P_{A+1} \\
          &= U^B D_{B-1} W_{B-1} \dots D_A W_A D P_2 D \dots  P_A D P_{A+1}.
    \end{align*}
    $W_i$ is Dyck path shifted to height $i$ by the definition of $D_i$. Hence, $W \in U^B (D\mathcal{P})^B$.
\end{Proof}

\begin{Prop}
		Let $A,B \in \Z_{\geq 0}$ such that $B<A$. The set $\mathcal{P}$ of Dyck paths avoiding down-run lengths in $\{Ar+B|r \in \Z_\geq{0}\}$ has the unambiguous grammar
		\[\mathcal{P}= \{EmptyPath\}\cup \underset{\underset{k \neq B}{1\leq k <A }}{\bigcup} 
		(U\mathcal{P})^{k-1} U D^k \mathcal{P} \cup (U\mathcal{P})^A D^A \mathcal{P}, \]
    and therefore
    \begin{equation*}
        P(z) = 1 + \sum_{\substack{0 \leq k < A \\ k \neq B}} z^k P^k(z) + z^A P^{A + 1}(z),
    \end{equation*}
    where $P(z)$ is the weight-enumerator of $\mathcal{P}$.
\end{Prop}

\begin{Proof} 
   It is obvious that the grammar uniquely parses the empty path, so let $P \in \mathcal{P}$ have length $n>0$. Let $D_0$ denote the first time $P$ returns to height $0$ and let $k$ be the length of the descending run in $P$ ending with the step $D_0$. Let $D_{k-1}...D_0$ denote this descending run.
   
   If $k<A$, then write $P=UWUD_{k-1}...D_0 P_0$. It is clear that $W$ is a walk from height $1$ to height $k-1$ and $P_0$ is a Dyck path, where both $W$ and $P_0$ have the same restrictions on descending runs as $P$.  Thus, $P_0 \in \mathcal{P}$ and, letting $U_i$ indicate the last up-step from height $i$ in $W$, we have
   \[W=P_1 U_1 P_2 U_2...P_{k-2}U_{k-2}P_{k-1}.\]
   By the definition of $U_{i}$, $P_j$ is a Dyck path shifted to height $j$ with the same restrictions on descending runs as $P$. This uniquely parses $P$ into the case $(U\mathcal{P})^{k-1} U D^k \mathcal{P}$.
   
   If the first descending run in $P$ that hits height zero has length $k \geq A$, then write \[ P=W D_{A-1}...D_0 P_0. \]
It is obvious that $P_0 \in \mathcal{P}$, and $W$ is a walk from height $0$ to height $A$ which never returns to height $0$. By an argument analogous to the previous paragraph, we can decompose $W$ as
\[W=U_0 P_1 U_1 P_2 U_2 P_3.. U_{A-1}P_{A},\]
 where $P_i \in \mathcal{P}.$ Thus, $W$ is of the form $(U \mathcal{P})^A$ and $P$ is uniquely parsed into the case $(U \mathcal{P})^A D^A \mathcal{P}$.
 
 This proves that $\mathcal{P}$ can be generated by the given grammar. It is clear that $(U\mathcal{P})^{k-1} U D^k \mathcal{P}$ is contained in $\mathcal{P}$ for $1 \leq k < A $ and $k\neq B$. $(U \mathcal{P})^A D^A \mathcal{P}$ is also contained in $\mathcal{P}$, since concatenating $D^A$ to a Dyck path in $\mathcal{P}$ does not change the length of any down-runs modulo $A$. The different cases defined on the right-hand side are clearly disjoint, so the grammar is unambiguous.
	
\end{Proof}

\begin{Prop}
		Let $A,B \in \Z_{\geq 0}$ such that $B\geq A$.  The set $\mathcal{P}$ of Dyck paths avoiding down-run lengths in $\{Ar+B|r \in \Z_\geq{0}\}$ satisfies the grammatical equation
		\[\mathcal{P} \cup (U\mathcal{P})^{B-1}U D^B \mathcal{P}	= \{EmptyPath\} \cup {\bigcup}_{1\leq k <A } 	(U\mathcal{P})^{k-1} U D^k \mathcal{P} \cup (U\mathcal{P})^A D^A \mathcal{P}. \]
    and therefore
    \begin{equation*}
        P(z) + z^B P^B(z) = 1 + \sum_{0 \leq k < A} z^k P^k(z) + z^A P^{A + 1}(z).
    \end{equation*}
    where $P(z)$ is the weight-enumerator of $\mathcal{P}$.
\end{Prop}

Note that the right-hand side is nearly identical to that of the previous claim -- the difference being that we can get paths in $(U \mathcal{P})^{B-1}U D^B \mathcal{P}$, which we will show below.

\begin{Proof}  
If $P$ is a path in $\mathcal{P}$, then we can uniquely parse $P$ into a case of the right hand side following the same argument given in the proof of Proposition 9.  Note that
\begin{align*}
(U \mathcal{P})^{B-1}U D^B \mathcal{P} &= (U \mathcal{P})^{A}
(U \mathcal{P})^{B-A-1}U D^{B-A} D^A \mathcal{P}\\
         &=(U \mathcal{P})^{A-1} U \{\mathcal{P} (U \mathcal{P})^{B-A-1} U D^{B-A} \} D^A \mathcal{P} 
         \end{align*}
         and the expression in brackets, $\mathcal{P} (U \mathcal{P})^{B-A-1} U D^{B-A}$, is contained in $\mathcal{P}$. Thus, any path in $(U \mathcal{P})^{B-1}U D^B \mathcal{P}$ is uniquely parsed into the case $(U \mathcal{P})^A D^A \mathcal{P}$. 
         
Thus, the left-hand side of the equation is generated by the right-hand side. The different cases defined on the right-hand side are also clearly disjoint.  It remains to show that all paths generated by the right-hand side are contained in the left-hand side.  It is clear that $(U\mathcal{P})^{k-1} U D^k \mathcal{P}$ is contained in $\mathcal{P}$ for $1 \leq k < A$.  For $W \in (U \mathcal{P})^A D^A \mathcal{P}$,
\[W=UP_1 U P_2 ... U P_A D^A P_{0}.\]
Let $\ell$ be the length of the last down-run in $P_A$. If $\ell \not \equiv B (\mod A)$, then $W$ contains no down-runs of lengths in $\{Ar+B|r \in \Z_\geq{0}\}$ and $W \in \mathcal{P}$. If $\ell \equiv B (\mod A)$, then $\ell \leq B-A$. When $\ell < B-A$, the corresponding down-run in $W$ has length < $B$, and again $W$ contains no down-runs of lengths in $\{Ar+B|r \in \Z_\geq{0}\}$. For $\ell = B-A$, write
\[P_A= W_A U_A W_{A+1} U_{A+1} W_{A+2}...W_{B}  U_{B} D^{B-A},\]
where $U_i$ is the last up-step from height $i$ in $P_A$, and thus $W_{i}$ is a Dyck path shifted to height $i.$
\begin{align*}
W&= UP_1 U P_2 ... U P_{A-1} U P_A D^A P_{0} \\
 &=UP_1 U P_2 ... U P_{A-1} U (W_A U_A... W_{B}U_{B} D^{B-A}) D^A P_{0}\\
 &= UP_1 U P_2 ... U P_{A-1} U W_A U_A... W_{B}U_{B} D^B P_{0}
         \end{align*}
and, hence, $W \in (U \mathcal{P})^B D^B \mathcal{P}.$ 
\end{Proof}

\begin{Prop}
    Let $r\in Z^+$. The set $\mathcal{P}$ of Dyck paths avoiding ascending and descending runs of lengths $L\in \{1,...,r\}$ satisfies the grammatical equation
    \begin{equation*}
        \mathcal{P} \cup UD \mathcal{P} =
            \{EmptyPath\} \cup U^{r+1}D^{r+1}\mathcal{P} \cup U\mathcal{P} D \mathcal{P}.
    \end{equation*}
    and therefore
    \begin{equation*}
        P(z) + zP(z) = 1 + z^{r + 1} P(z) + z P^2(z),
    \end{equation*}
    where $P(z)$ is the weight-enumerator of $\mathcal{P}$.
	
\end{Prop}

\begin{Proof}
	If $P\in \mathcal{P}$ is the empty path, then the grammar uniquely parses $P$. Otherwise, $P\in \mathcal{P}$ must begin with an ascending run of length $\ell >r$. If $\ell=r+1$, then clearly $U^{r+1}$ must be immediately followed by the descending run $D^{r+1}$, and $P$ is uniquely parsed into the case $U^{r+1}D^{r+1} \mathcal{P}$.

    If $\ell>r+1$, then let $D_0$ denote the step where $P$ returns to height $0$ for the first time and write
    \[P=U P_1 D_0 P_2.\]
    It is obvious that $P_2 \in \mathcal{P}$ and
    $P_1$ is a Dyck path shifted to height 1. By restrictions on $P$, the final descending run in $P_1$ must have length $L \geq r$. If $L=r$ then the preceding ascending run ends at height $r+1$. But the ascending runs in $P$ must have length of at least $r+1$, and hence $P_1$ hits height $0$, contradicting the definition of $D_0$. From here, it is clear that $P_1$ has the same restrictions on ascending and descending runs as $P$. Thus, $P$ is uniquely parsed into the case $U \mathcal{P} D \mathcal{P}$. 
    
    Since it is trivial that $UD\mathcal{P}$ is contained in $U\mathcal{P} D \mathcal{P}$, we have shown that the left-hand side of the given equation is generated by the right-hand side. It is also obvious that the cases defined on the right-hand side are disjoint and that
   $ \{EmptyPath\} \cup U^{r+1}D^{r+1} \mathcal{P}$ is contained in $\mathcal{P}$. A path $UP_1DP_2\in U \mathcal{P}D\mathcal{P}$ is contained in $UD\mathcal{P}$ if $P_1$ is the empty path and $\mathcal{P}$ otherwise. Thus, $\mathcal{P}$ satisfies the given grammatical equation.

\end{Proof}

\begin{Prop}
    Let $m,n \in \Z^+$. The set $\mathcal{P}$ of Dyck paths avoiding ascending runs of lengths in $\{1,...,m\}$ and descending runs of lengths in $\{1,...,n\}$ satisfies the "grammatical equation"
        \begin{equation*}
            \mathcal{P}  \cup UD\mathcal{P} = \{EmptyPath\} \cup U\mathcal{P}D \mathcal{P} \cup U^{m+1}D^{n+1}(\mathcal{P}D)^{m-n} \mathcal{P}, \text{ if $m \geq n$} \tag{1}
        \end{equation*}
        
        \begin{equation*}
        \mathcal{P}  \cup UD\mathcal{P} = \{EmptyPath\} \cup U\mathcal{P}D \mathcal{P} \cup
        (U \mathcal{P} )^{n-m} U^{m+1}  D^{n+1} \mathcal{P}  , \text{ if $m \leq n.$} \tag{2}
        \end{equation*}
\end{Prop}
 
 \begin{Proof}
 	We have already shown that this statement is true for $m=n$. Suppose $m>n$. If $P\in\mathcal{P}$ is the empty path, then the grammar uniquely parses $P$. Otherwise, $P$ must begin with an ascending run of length $\ell > m$.  If $\ell=m+1$ then $U^{m+1}$ is followed by a descending chain of length of at least $n+1$. Let $D_i$ denote the first time $P$ returns to height $i$ for $0 \leq i \leq m-n-1$, and write
 	\[P=U^{m+1}D^{n+1} P_{m-n} D_{m-n-1}...P_1 D_0 P_0.\]
 	It is obvious that $P_i$ is a Dyck path, shifted to height $i$, that has the same restrictions on ascending runs and descending runs (with the exception of the final descending run) as $P$. Since $P_i$ is a Dyck path, its final descending run must be at least as long as the ascending run preceding it. Thus, $P_i$ is either the empty path or ends with a descending run of length $L >m >n$. Thus, $P$ is uniquely parsed into the case $U^{m+1}D^{n+1} (\mathcal{P} D)^{m-n} \mathcal{P}.$
 	
 	If $\ell > m+1$ then, letting $D_0$ denote the first time $P$ returns to height $0$, write
  \[P=U P_1 D_0 P_0.\]
  Clearly, $P_0\in \mathcal{P}$, and $P_1$ is a Dyck path shifted to height $1$ and has the same restrictions on ascending runs as $P$. Using the same argument as for $P_i$ in the previous case, the descending runs in $P_1$ also have the same restrictions as $P$. This uniquely parses $P$ into the case $U\mathcal{P}D \mathcal{P}$. Finally, it is obvious that $UD\mathcal{P}$ is contained in $U \mathcal{P} D\mathcal{P}$, so the left-hand side of $(1)$ is generated by the right-hand side. 

  Now, suppose $m<n$ and let $P$ be an element in $\mathcal{P}.$ Let $L$ denote the length of the descending run where $P$ returns to height $0$ for the first time. If $L=n+1$, then write
  \[P= W U^{m+1} D^{n+1} P_0,\]
  where $W$ is a walk from height $0$ to $n-m$ with the same restrictions on ascending and descending runs as $P$ and $P_0 \in \mathcal{P}$. Decomposing $W$ and letting $U_i=$ denote the last time $W$ leaves height $i$, write
  \[W= U_0 P_1 U_1 P_2 ... U_{m-n-1} P_{m-n}.\]
  Then, for all $i$, $P_i$ is clearly a Dyck path with the same restrictions on descending runs and ascending runs (with the exception of the first run) as $P$. The first ascending run in $P_i$ must be longer than the descending run that follows it, which has length of at least $n+1>n$. Thus, $P_i \in \mathcal{P}$ and $P$ has the grammar $(U\mathcal{P})^{n-m} U^{m+1}D^{n+1}\mathcal{P}$.
  
  If $L>n+1$, then write
  \[P=U P_1 D P_2\],
  where $D$ denotes the first time $P$ returns to height $0$. Then, using the same argument as we gave when $\ell>m+1$, $P$ is parsed into the case $U \mathcal{P} D \mathcal{P}.$ Since $UD\mathcal{P}$ is obviously contained in $U\mathcal{P} D \mathcal{P}$, we have shown that the left-hand side of $(2)$ is generated by the right-hand side.
  
  In both $(1)$ and $(2)$, it is clear that the cases on the right-hand side are disjoint and the empty path is an element of $\mathcal{P}.$ Also,  $U P_1 D P_2 \in U\mathcal{P}D\mathcal{P}$ is contained in  $\mathcal{P}$ if $P_1$ is not the empty path, and is contained in $UD\mathcal{P}$ otherwise. In $(1)$, $U^{m+1}D^{n+1}(\mathcal{P}D)^{m-n}\mathcal{P}$ is contained in $\mathcal{P}$, since all ascending runs clearly avoid restrictions on $\mathcal{P}$ and the descending runs are formed by concatenating down-steps to descending runs of length of at least $n-1$. Similarly in $(2)$, we have that $(U\mathcal{P})^{n-m} U^{m+1} D^{n+1} \mathcal{P}$ is contained in $\mathcal{P}$. Thus, $\mathcal{P}$  satisfies the given grammatical equation in both cases.
   
\end{Proof}

\begin{Prop}
        Let $r,k \in \Z^+$ and let $\mathcal{P}$ be the set of Dyck paths
        avoiding ascending runs of length $\{1,...,r\}$ and descending runs of
        length $\{k+1,...,r\}$. Then the 'grammar` of $\mathcal{P}$ is 
        \begin{equation*}
            \mathcal{P}  \cup UD\mathcal{P} \cup U^{r+1}D^{k} (D\mathcal{P})^{r+1-k}
                = \{EmptyPath\} \cup U \mathcal{P}D \mathcal{P}
                    \cup U^{r+1}D^{r+1}\mathcal{P} \cup U^{r+1} (DP)^{r+1}
        \end{equation*}
\end{Prop}

\begin{Proof}
If $P\in\mathcal{P}$ is the empty path, then the grammar uniquely parses $P$. Otherwise, $P$  begins an ascending run of length $\ell>r$, and we can deduce that it also ends with a descending run of length $L>r$. If $\ell>r+1$, then let $D_0$ denote the first time that $P$ returns to the $x-$axis and write
\[P=U P_1 D_0 P_0.\]
It is easy to see that $P_0$ is a path in $\mathcal{P}$ and $P_1$ is a Dyck path shifted to height 1. The initial ascending run in $P_1$ has length $\ell-1>r$. Thus, all ascending runs in $P_1$ have length of at least $r+1$ and, since $P_1$ is a shifted Dyck path, the final descending run in $P_1$ must also have length of at least $r+1$. From here, it is easy to see that $P_1$ has the same restrictions on ascending and descending runs as $P$. $P$ is therefore uniquely parsed into the case $U \mathcal{P}D \mathcal{P}.$

Suppose $\ell=r+1$. Let $D_i$ be the step where $P$ returns to height $i$ for the first time and write
\[P=U^{r+1}D_r P_r ... D_0 P_0.\]
$P_i$ is a Dyck path for all $i$ and, if $P_i$ is not the empty path, it must end with a descending run of length $r+1$ by restrictions on ascending runs. Thus $P_i$ is a path in $\mathcal{P}$,  and $P$ is parsed into the case $U^{r+1}(D \mathcal{P})^{r+1}$.

It is trivial that $UD\mathcal{P}$ is contained $U \mathcal{P} D \mathcal{P}$ and $U^{r+1} D^k (D \mathcal{P})^{r+1-k}$ is contained in $U^{r+1}(D \mathcal{P})^{r+1}$. Thus, the left-hand side is generated by the right-hand side. Note that, on the left-hand side,
\[UD\mathcal{P} \cap \mathcal{P} = UD\mathcal{P} \cap  U^{r+1} D^k (D \mathcal{P})^{r+1-k}=\emptyset, \]
however
\[\mathcal{P} \cap  U^{r+1} D^k (D \mathcal{P})^{r+1-k}=U^{r+1}D^{r+1} \mathcal{P}.\]

Looking at the right-hand side, it is clear that $\{EmptyPath\}, U \mathcal{P} D \mathcal{P}$, and $U^{r+1} (D\mathcal{P})^{r+1}$ are disjoint, and $U^{r+1}D^{r+1}\mathcal{P}$ is contained in $U^{r+1} (D\mathcal{P})^{r+1}$. Note that this resolves the issue of double counting paths in  $U^{r+1}D^{r+1} \mathcal{P}$ on the left-hand side. Thus, all that remains to show is that all the paths generated by the right-hand side are contained in the left-hand side.

The path $U P_1 D P_0$ in $U \mathcal{P} D \mathcal{P}$ is clearly in $\mathcal{P}$ if $P_1$ is not the empty path and in $UD\mathcal{P}$ otherwise. For $W$ in $U^{r+1}(D \mathcal{P})^{r+1}$, write
 \[W= U^{r+1}D_r P_r...D_1 P_1 D_0 P_{0}.\]
 Choose the smallest $i$ such that $P_{r-i}$  is not the empty path or, if no such $i$ exists, set $i=r$. Then
the first descending run in $W$ has length $i+1$.  If $i\geq k$ then $W$ is an element of $U^{r+1} D^k (D \mathcal{P})^{r+1-k}$. Otherwise, we claim that $W$ is a path in $\mathcal{P}$.  It is clear that $W$ is a Dyck path and we have seen that nonempty $P_j \in \mathcal{P}$ must end in a descending run of length of at least $r+1$. Thus, we only need to show that the first descending run in $W$ follows the restrictions in $\mathcal{P}.$ This is clearly true since $i<k$. Hence $W\in \mathcal{P}$, and $\mathcal{P}$  satisfies the  grammatical equation as desired.
\end{Proof}

\section{Conclusion}%
\label{sec:conclusion}

We have given several grammatical proofs of various combinatorial results (some lifted from \cite{z}) and established some infinite families of grammars.
Our methods work because we are able to derive \emph{context-free grammars}
describing certain restricted classes Dyck paths, namely when our restrictions involved sets of arithmetic progressions. It is natural to ask if context-free grammars exist for other types of restrictions, but this is beyond our current scope.

\end{document}